\documentclass{amsart}
\usepackage{latexsym,amssymb,amsthm,amsmath}

\usepackage{amssymb}
\usepackage{amsmath}

\theoremstyle{plain}
\newtheorem{theorem}{Theorem}[section]
\newtheorem*{Theorem B}{Theorem B}
\newtheorem*{Theorem A}{Theorem A}
\newtheorem{lemma}{Lemma}[section]

\numberwithin{equation}{section}

\theoremstyle{remark}
\newtheorem{remark}{Remark}[section]

\title{Bi-warped product submanifolds of nearly Kaehler manifolds}

\author{Siraj Uddin}
\address{S. Uddin: Department of Mathematics, Faculty of Science, King Abdulaziz University, 21589 Jeddah, Saudi Arabia}
\email{siraj.ch@gmail.com}
\author{Bang-Yen Chen}
\address{B.-Y. Chen: Department of Mathematics, Michigan State University, 619 Red Cedar Road,   East Lansing, Michigan 48824--1027, U.S.A.}
\email{bychen@math.msu.edu}
\author{Awatif AL-Jedani}
\address{A. AL-Jedani: Department of Mathematics, Faculty of Science, King Abdulaziz University, 21589 Jeddah, Saudi Arabia}
\email{awtfmm@hotmail.com}
\author{Azeb Alghanemi}
\address{A. Alghanemi: Department of Mathematics, Faculty of Science, King Abdulaziz University, 21589 Jeddah, Saudi Arabia}
\email{aalghanemi@kau.edu.sa}

\begin{document}
\begin{abstract} 
We study bi-warped product submanifolds of nearly Kaehler manifolds which are the natural extension of warped products. We prove that every bi-warped product submanifold of the form $M=M_T\times_{f_1}\! M_\perp\times_{f_2}\! M_\theta$ in a nearly Kaehler manifold satisfies the following sharp inequality: $$\|h\|^2\geq 2p\|\nabla (\ln f_1)\|^2+4q\left(1+{\small \frac{10}{9}}\cot^2\theta\right)\|\nabla(\ln f_2)\|^2,$$ where $p=\dim M_\perp$, $q=\frac{1}{2}\dim M_\theta$, and $f_1,\,f_2$ are smooth positive functions on $M_T$. We also investigate the equality case of this inequality. Further,  some applications of this inequality are also given.
\end{abstract}

\maketitle

\sloppy
\section{Introduction}
Bi-warped product manifolds are natural extensions of (ordinary) warped product and Riemannian product manifolds. 
Let $M_0,M_1$ and $ M_2$ be Riemannian manifolds and $M=M_0\times M_1\times M_2$ be the Cartesian product of $M_0,\,M_1$ and $M_2$. For each $i=0,1,2$,  we denote by $\pi_i: M\to M_i$ the canonical projection of $M$ onto $M_i$. For each $\pi_i: M\to M_i$, let $\pi_{i^*}$ denote the corresponding tangent map $\pi_{i^*}: TM\to T M_i$. Denote by $\Gamma(TM)$ the Lie algebra of vector fields of $M$.

If $f_1,\,f_2$ are positive real valued functions on $M_0$, then 
\begin{align*}
&g(X, Y)=g(\pi_{0*}X, \pi_{0*}Y)+\left(f_1\circ\pi_1\right)^2g(\pi_{1*}X, \pi_{1*}Y)+\left(f_2\circ\pi_2\right)^2g(\pi_{2*}X, \pi_{2*}Y),
\\& \hskip1.6in  X,Y\in \Gamma(TM),
\end{align*}
defines a Riemannian metric on $M_0\times M_1\times M_2$, called a \textit{bi-warped product metric}. The product manifold $M=M_0\times M_1\times M_2$ endowed with this warped product metric $g$, denoted by $M_1\times_{f_1}\! M_2\times_{f_2}\! M_3$, is called a {\it bi-warped product manifold}. The functions $f_1,\,f_2$ are  called the {\it warping functions}. 
Obviously, if  $f_1,\,f_2$ are both constant, $M$ is simply a Riemannian product; and if exactly one of $f_1,f_2$ is constant, then $M$ is an (ordinary) warped product manifold. Further, if none of $f_1,f_2$ is constant, then $M$ is called a {\it proper bi-warped product manifold}.

Let $M=M_0\times_{f_1}M_1\times_{f_2}M_2$ be a bi-warped product submanifold.  We put 
$${\mathfrak{D}}=TM_T,\;\; {\mathfrak{D}}^\perp=TM_\perp,\;\; {\mathfrak{D}}^\theta=TM_\theta,\;\; N= {}_{f_1}M_1\times_{f_2}M_2.$$
Then we have  (cf. \cite{CD} and \cite{U6})
\begin{align}\label{warped}
\nabla_XZ=\sum_{i=1}^2\left(X(\ln f_i)\right)Z^i,
\end{align}
for  $X\in {\mathfrak{D}}_0$ and $Z\in\Gamma(TN)$, where $\nabla$ is the Levi-Civita connection on $M$ and $Z^i$ (i=1,2) is the $M_i$-component of $Z$.

Nearly Kaehler manifolds, also known as  almost Tachibana manifolds, were first studied in 1959 by S. Tachibana  \cite{Tach59} and then in 1970 by A. Gray \cite{Gr70}.
Obviously, Kaehler manifolds are nearly Kaehler, but the converse is not true. Non-Kaehlerian nearly Kaehler manifolds are called {\it strict nearly Kaehler manifolds}.

The best known example of a strict nearly Kaehler manifold is the unit 6-sphere $S^6$. 
More general examples are  homogeneous spaces $G/K$, where $G$ is a compact semisimple Lie group and $K$ is the fixed point set of an automorphism of $G$ of order 3 (cf. \cite{WG68}).  In 1985,  T. Friedrich and R. Grunewald proved in \cite{FG85} that a Riemannian 6-manifold is nearly Kaehler if and only if admits a Riemannian Killing spinor. After then, strict nearly Kaehler manifolds obtained a lot of attentions due to their relation to Killing spinors.

The notion of warped products plays very important roles not only in geometry but also in mathematical physics, especially in general relativity. The term of ``warped product'' was introduced by R. L. Bishop and B. O'Neill in \cite{Bi69}, who used it to construct a large class of complete manifolds of negative curvature. Inspired by Bishop and O'Neill's article,  many important works on warped products from {\it intrinsic} point of view were done during the last fifty years.

On the other hand, the study of warped product submanifolds from {\it extrinsic} point of review was initiated around the beginning  of this century  in \cite{C3,C4,C5}. Since then warped product submanifolds have became an active research subject (see, e.g., \cite{C6,book17,C7,CD,CU,sahin06,sahin13,SC}). For instance, B. Sahin  studied in \cite{sahin13} warped product pointwise semi-slant submanifolds in Kaehler manifolds. H. M. Tastan \cite{Tas17} extended this study to bi-warped product submanifolds in Kaehler manifolds by considering that one of the fiber of warped product is a pointwise slant submanifold.

In this article, we study bi-warped product submanifolds in nearly Kaehler manifolds. In section 2, we give basic definitions and formulas. In section 3, we prove some useful results for the proof of our main result. In section 4, we prove a sharp inequality for bi-warped product submanifolds in nearly Kaehler manifolds. We also discuss the equality case of the inequality. In the last section, we provide some applications of our main result.


\section{Preliminaries}
An even-dimensional differentiable manifold $N_K$ with Riemannian metric $g$ and almost complex structure $J$ is called a {\it nearly Kaehler manifold} if  (cf. \cite{book11,Gr70})
\begin{equation}\begin{aligned}
 \label{metric}
&g(JX,JY)=g({X}, {Y}),\,\,\; (\tilde\nabla_XJ)Y+(\tilde\nabla_YJ)X=0,
\end{aligned}\end{equation}
for any vector fields $X, Y\in \Gamma(TN_K)$. 

Let $M$ be a submanifold of a Riemannian manifold  $\tilde M$ with induced metric $g$. Let  $\Gamma(T^\perp M)$ denote the set of all vector fields normal to $M$. Then the Gauss and Weingarten formulas are given respectively by (see, for instance, \cite{C3,book17})
\begin{align}
 \label{gauss} &\tilde\nabla_XY=\nabla_XY+h(X, Y),
\\& \label{wen}\tilde\nabla_X\xi=-A_{\xi} X+\nabla^{\perp}_X\xi,
\end{align}
 for  vector fields $X,~Y\in\Gamma(TM)$ and $\xi \in\Gamma(T^{\perp}M)$, where $\nabla$ and $\nabla^{\perp}$ denote the induced connections on the tangent and normal bundles of $M$, respectively, and $h$ is the second fundamental form, $A$ is the shape operator of the submanifold. The second fundamental form $h$ and the shape operator $A$ are related by 
\begin{align}
 \label{wen1}
g(h(X, Y), N)=g(A_NX, Y).
\end{align}

For an $n$-dimensional submanifold $M$ of an almost Hermitian $2m$-manifold $\tilde M$, we choose a local orthonormal frame field $\{e_1,\cdots, e_n, e_{n+1},\cdots, e_{2m}\}$ such that, restricted to $M$,  $e_1,\cdots, e_n$ are tangent to $M$ and $e_{n+1},\cdots, e_{2m}$ are normal to $M$. 

Let $\{h^r_{ij}\},\,1\leq i, j\leq n;\,n+1\leq r\leq 2m,$ denote the coefficients of the second fundamental form $h$ with respect to the local frame field. Then, we have
\begin{equation}\begin{aligned}
& \label{funda}
h_{ij}^r=g(h(e_i, e_j), e_r)=g(A_{e_r}e_i, e_j),\,\,\,\\& \|h\|^2=\sum_{i,j=1}^ng(h(e_i, e_j), h(e_i, e_j)).
\end{aligned}\end{equation}

For any $X\in\Gamma(TM)$, we put
\begin{align}
 \label{tan}
JX = T X +FX,
\end{align}
where $TX$ and $FX$ are the tangential and normal components of $JX$, respectively.
 A submanifold $M$ of an almost Hermitian manifold $\tilde M$ is said to be {\it {holomorphic}} (resp. {\it {totally real}}) if $J(T_pM)=T_pM$ (resp. $J(T_pM)\subseteq T_p{^\perp}M)\,\forall~p\in M$.

There are other important classes of submanifolds determined by the behaviour of almost complex structure $J$ acting on the tangent space of $M$:
 For a nonzero vector $X\in T_pM$, $p\in M$, the angle $\theta(X)$ between $JX$ and $T_pM$ is called the Wirtinger angle of $X$. A submanifold $M$ is said to be {\it slant} (cf. \cite{C1,C2}) if the Wirtinger angle $\theta(X)$ is constant on $M$, i.e., it is independent of the choice of $X\in T_pM$ and $p\in M$. In this case, $\theta$ is called the {\it slant angle} of $M$. Holomorphic and totally real submanifolds are slant submanifolds with slant angles $0$ and $\frac{\pi}{2}$, respectively. A slant submanifold is called {\it proper slant} if it is neither holomorphic nor totally real.
  More generally, a distribution ${\mathfrak{D}}$ on $M$ is called a {\it slant distribution} if the angle $\theta(X)$ between $JX$ and ${\mathfrak{D}}_p$ is independent of the choice of  $p\in M$ and of $0\ne X\in{\mathfrak{D}}_p$. 

It is well-known from \cite{C1} that a submanifold $M$ of an almost Hermitian manifold $\tilde M$ is slant if and only if we have
\begin{align}
 \label{slant}
T^2X = -(\cos^2\theta)X,\;\; X\in\Gamma(TM).
\end{align}
From \eqref{slant} we have the following.
\begin{align}
 \label{slant1}
g(TX,TY)=(\cos^2\theta) g(X,Y),
\end{align}
\begin{align}
 \label{slant2}
g(FX, FY) = (\sin^2\theta) g(X, Y),
\end{align}
for any vector fields $X,Y$ tangent to $M$.


\section{Bi-warped product submanifolds }

Now, we study bi-warped product submanifolds  in a nearly Kaehler manifold $\tilde M$ which are of the form $M=M_T\times_{f_1}\! M_\perp\times_{f_2}\! M_\theta$, where $M_T,\, M_\perp,M_\theta$ are holomorphic, totally real and proper slant submanifolds of $\tilde M$, respectively. 
If we put 
$${\mathfrak{D}}=TM_T,\;\; {\mathfrak{D}}^\perp=TM_\perp,\;\; {\mathfrak{D}}^\theta=T M_\theta,$$
then the tangent and normal bundles of $M$ are decomposed as
\begin{align*}
TM={\mathfrak{D}}\oplus{\mathfrak{D}}^\perp\oplus{\mathfrak{D}}^\theta,\,\,\,T^\perp M=J{\mathfrak{D}}^\perp\oplus F{\mathfrak{D}}^\theta\oplus\mu
\end{align*}
where $\mu$ is an $j$-invariant normal subbundle of the normal bundle $T^\perp M$.
From now one, we use the following conventions:  $X_1, Y_1,\ldots $ are vector fields in $\Gamma({\mathfrak{D}})$ and $X_2, Y_2,\ldots $ are vector fields in $\Gamma({\mathfrak{D}}^\theta)$, while  $Z, W,\ldots $ are vector fields in $\Gamma({\mathfrak{D}}^\perp)$.
\vskip.1in

We present the following useful results for later use.

\begin{lemma}\label{L1} Let $M=M_T\times_{f_1}\! M_\perp\times_{f_2}\! M_\theta$ be a bi-warped product submanifold of a nearly Kaehler manifold $\tilde M$. Then we have
\begin{enumerate}
\item[(i)] $g(h(X_1, Y_1), JZ)=0,$
\item [(ii)]  $g(h(X_1, Y_1), FX_2)=0,$
\item [(iii)] $g(h(X_1, Z), JW)=-JX_1(\ln f_1)\,g(Z, W),$
\end{enumerate}
for any $X_1,Y_1\in \Gamma({\mathfrak{D}})$, $Z,W\in \Gamma({\mathfrak{D}}^\perp)$ and $X_2\in \Gamma({\mathfrak{D}}^\theta)$.
\end{lemma}
\begin{proof}
For any $X_1,Y_1\in \Gamma({\mathfrak{D}})$ and $Z\in \Gamma({\mathfrak{D}}^\perp)$, we have
\begin{align*}
g(h(X_1, Y_1), JZ)=g(\tilde\nabla_{X_1}Y_1, JZ)=g((\tilde\nabla_{X_1}J)Y_1, Z)-g(\tilde\nabla_{X_1}\!JY_1, Z).
\end{align*}
Using \eqref{warped}, we find
\begin{align*}
g(h(X_1, Y_1), JZ)=g((\tilde\nabla_{X_1}\! J)Y_1, Z)+X_1(\ln f_1)g(JY_1, Z).
\end{align*}
By the orthogonality of vector fields, we have
\begin{align}\label{3.1}
g(h(X_1, Y_1), JZ)=g((\tilde\nabla_{X_1}\! J)Y_1, Z).
\end{align}
Interchanging $X_1$ by $Y_1$ in \eqref{3.1}, we get
\begin{align}\label{3.2}
g(h(X_1, Y_1), JZ)=g((\tilde\nabla_{Y_1}\! J)X_1, Z).
\end{align}
Then, first part follows from \eqref{3.1} and \eqref{3.2} by using \eqref{metric}. In a similar fashion, we can prove (ii). For the third part, we have
\begin{align*}
g(h(X_1, Z), JW)=g(\tilde\nabla_{Z}X_1, JW)=g((\tilde\nabla_{Z} J)X_1, W)-g(\tilde\nabla_{Z}\! JX_1, W),
\end{align*}
for any $X_1\in \Gamma({\mathfrak{D}})$ and $Z, W\in \Gamma({\mathfrak{D}}^\perp)$. Again, using \eqref{warped} and \eqref{metric}, we derive
\begin{align*}
g(h(X_1, Z), JW)&=-g((\tilde\nabla_{X_1}\! J)Z, W)-JX_1(\ln f_1)g(Z, W)\\
&=-g(\tilde\nabla_{X_1}\! JZ, W)+g(J\tilde\nabla_{X_1}Z, W)-JX_1(\ln f_1)g(Z, W).
\end{align*}
Using \eqref{metric}, \eqref{gauss}, \eqref{wen} and \eqref{wen1}, we get
\begin{align}\label{3.3}
2g(h(X_1, Z), JW)=g(h(X_1, W), JZ)-JX_1(\ln f_1)g(Z, W).
\end{align}
Interchanging $Z$ by $W$ in \eqref{3.3}, we obtain
\begin{align}\label{3.4}
2g(h(X_1, W), JZ)=g(h(X_1, Z), JW)-JX_1(\ln f_1)g(Z, W).
\end{align}
Hence, the third part follows from \eqref{3.3} and \eqref{3.4}, which proves the lemma.
 \end{proof}
 
A bi-warped product submanifold $M=M_T\times_{f_1}\! M_\perp\times_{f_2}\! M_\theta$ in a nearly Kaehler manifold $\tilde M$ is said to be {\it ${\mathfrak{D}}\oplus{\mathfrak{D}}^\perp$--mixed totally geodesic} (resp., {\it ${\mathfrak{D}}\oplus{\mathfrak{D}}^\theta$--mixed totally geodesic}) if its second fundamental $h$ satisfies
 \begin{align}\notag &\hskip.3in  h(X_1,Z)=0\;\; \; \forall X_1\in \Gamma({\mathfrak{D}}), \;\; \forall Z\in \Gamma({\mathfrak{D}^\perp})
 \\&\notag (resp.,\;\; h(X_1,X_2)=0\;\; \; \forall X_1\in \Gamma({\mathfrak{D}}), \;\; \forall X_2\in \Gamma({\mathfrak{D}^\theta}))
 .\end{align}
 
 \begin{lemma}\label{L2} Let $M=M_T\times_{f_1}\! M_\perp\times_{f_2}\! M_\theta$ be a bi-warped product submanifold of a nearly Kaehler manifold $\tilde M$. Then we have
\begin{enumerate}
\item [(i)]  $g(h(X_1, Z), FX_2)=\frac{1}{2}\,g(h(X_1, X_2), JZ)=0,$
\item [(ii)] $g(h(X_1, X_2), FY_2)=\frac{1}{3}X_1(\ln f_2)\,g(TX_2, Y_2)-JX_1(\ln f_2)g(X_2, Y_2),$
\end{enumerate}
for any $X_1\in \Gamma({\mathfrak{D}})$, $Z\in \Gamma({\mathfrak{D}}^\perp)$ and $X_2,Y_2\in \Gamma({\mathfrak{D}}^\theta)$.
\end{lemma}
 \begin{proof} For any $X_1\in \Gamma({\mathfrak{D}}),\,\,Z\in \Gamma({\mathfrak{D}}^\perp)$, and $X_2\in \Gamma({\mathfrak{D}}^\theta)$, we have
 \begin{align*}
g(h(X_1, Z), FX_2)&=g(\tilde\nabla_{Z}X_1, JX_2-TX_2)\\
&=g((\tilde\nabla_{Z}J)X_1, X_2)-g(\tilde\nabla_{Z}JX_1, X_2)-g(\tilde\nabla_{Z}X_1, TX_2).
\end{align*}
Using \eqref{metric}, \eqref{warped} and the orthogonality of vector fields, we derive
 \begin{align*}
g(h(X_1, Z), FX_2)=-g((\tilde\nabla_{X_1}J)Z, X_2)=-g(\tilde\nabla_{X_1}JZ, X_2)+g(J\tilde\nabla_{X_1}Z, X_2).
\end{align*}
Then, from \eqref{metric}-\eqref{wen1}, we obtain
\begin{align}\label{3.5}
g(h(X_1, Z), FX_2)=\frac{1}{2}\,g(h(X_1, X_2), JZ),
\end{align}
Which is the first equality of (i). 

On the other hand, we have
\begin{align*}
g(h(X_1, X_2), JZ)=g(\tilde\nabla_{X_2}X_1, JZ)=g((\tilde\nabla_{X_2}J)X_1, Z)-g(\tilde\nabla_{X_2}JX_1, Z).
\end{align*}
Using \eqref{metric}, \eqref{warped} and the orthogonality of vector fields, we find
\begin{align*}
g(h(X_1, X_2), JZ)=-g((\tilde\nabla_{X_1}J)X_2, Z)=-g(\tilde\nabla_{X_1}JX_2, Z)+g(J\tilde\nabla_{X_1}X_2, Z).
\end{align*}
Then it follows from \eqref{metric} and \eqref{tan} that
\begin{align*}
g(h(X_1, X_2), JZ)=-g(\tilde\nabla_{X_1}TX_2, Z)-g(\tilde\nabla_{X_1}FX_2, Z)-g(\tilde\nabla_{X_1}X_2, FZ).
\end{align*}
Again, using \eqref{warped}, \eqref{gauss}-\eqref{wen1} and the orthogonality of vector fields, we obtain
\begin{align}\label{3.6}
g(h(X_1, X_2), JZ)=\frac{1}{2}\,g(h(X_1, Z), FX_2).
\end{align}
Hence, the second equality of (i) follows from \eqref{3.5} and \eqref{3.6}. For, the second part of the lemma, we have
\begin{align*}
g(h(&X_1, X_2), FY_2)=g(\tilde\nabla_{X_2}X_1, JY_2-TY_2)\\
&=g((\tilde\nabla_{X_2}J)X_1, Y_2)-g(\tilde\nabla_{X_2}JX_1, Y_2)-g(\tilde\nabla_{X_2}X_1, TY_2)\\
&=-g((\tilde\nabla_{X_1}J)X_2, Y_2)-JX_1(\ln f_2)g(X_2, Y_2)-X_1(\ln f_2)g(X_2, TY_2)\\
&=-g(\tilde\nabla_{X_1}TX_2, Y_2)-g(\tilde\nabla_{X_1}FX_2, Y_2)-g(\tilde\nabla_{X_1}X_2, TY_2)\\
&-g(\tilde\nabla_{X_1}X_2, FY_2)-JX_1(\ln f_2)g(X_2, Y_2)-X_1(\ln f_2)g(X_2, TY_2).
\end{align*}
Using \eqref{gauss}-\eqref{wen1} and \eqref{warped}, we find
\begin{equation}\begin{aligned}\label{3.7}
2g(h(X_1, X_2), FY_2)=\;&g(h(X_1, Y_2), FX_2)-X_1(\ln f_2) g(X_2, TY_2)
\\&-JX_1(\ln f_2) g(X_2, Y_2).
\end{aligned}\end{equation}
Interchanging $X_2$ by $Y_2$ in \eqref{3.7}, we get
\begin{equation}\begin{aligned}\label{3.8}
2g(h(X_1, Y_2), FX_2)=\;&g(h(X_1, X_2), FY_2)+X_1(\ln f_2) g(X_2, TY_2)
\\
&-JX_1(\ln f_2) g(X_2, Y_2).
\end{aligned}\end{equation}
The second part follows from \eqref{3.7} and \eqref{3.8}. Hence the proof is complete.
\end{proof}

The following relations are easily obtained by interchanging $X_1$ by $JX_1$ and $X_2$  and $Y_2$ by $TX_2$ and $TY_2$, respectively.
\begin{align}\label{3.9}
g(h(X_1, X_2), FTY_2)=\frac{1}{3}X_1(\ln f_2)\cos^2\theta\,g(X_2, Y_2)-JX_1(\ln f_2)g(X_2, TY_2),
\end{align}
\begin{align}\label{3.10}
g(h(JX_1, X_2), FTY_2)&=\frac{1}{3}JX_1(\ln f_2)\cos^2\theta\,g(X_2, Y_2)\notag\\
&+X_1(\ln f_2)g(X_2, TY_2),
\end{align}
\begin{align}\label{3.11}
g(h(X_1, TX_2), FY_2)&=-\frac{1}{3}X_1(\ln f_2)\cos^2\theta\,g(X_2, Y_2)\notag\\
&-JX_1(\ln f_2)g(TX_2, Y_2),
\end{align}
\begin{align}\label{3.12}
g(h(JX_1, TX_2), FY_2)&=-\frac{1}{3}JX_1(\ln f_2)\cos^2\theta\,g(X_2, Y_2)\notag\\
&+X_1(\ln f_2)g(TX_2, Y_2),
\end{align}
\begin{align}\label{3.13}
g(h(X_1, TX_2), FTY_2)&=-\frac{1}{3}X_1(\ln f_2)\cos^2\theta\,g(X_2, TY_2)\notag\\
&-JX_1(\ln f_2)\cos^2\theta g(X_2, Y_2).
\end{align}
 
 From Lemma \ref{L1}(iii) we obtain immediately the following.
 
  \begin{theorem}\label{T1}
Let $M=M_T\times_{f_1}\! M_\perp\times_{f_2}\! M_\theta$ be a bi-warped product submanifold of a nearly Kaehler manifold $\tilde M$. If $M$ is ${\mathfrak{D}}\oplus{\mathfrak{D}}^\perp$--mixed totally geodesic, then $f_1$ is constant, and hence $M$ is an ordianary warped product manifold.
\end{theorem}
  
Similarly, from Lemma \ref{L2} (ii), we may obtain the following.

\begin{theorem}\label{T2} Let $M=M_T\times_{f_1}\! M_\perp\times_{f_2}\! M_\theta$ be a proper bi-warped product submanifold of a nearly Kaehler manifold $\tilde M$. If $M$ is  ${\mathfrak{D}}\oplus {\mathfrak{D}}^\theta$--mixed totally geodesic, then $f_2$ is constant on $M$. \end{theorem}
 \begin{proof} From Lemma \ref{L2} (ii) and \eqref{3.10}, we have
 \begin{equation}\begin{aligned}\label{3.14}&
\left(\cos^2\theta-9\right)JX_1(\ln f_2)\,g(X_2, Y_2)
\\&\hskip.3in = 9g(h(X_1, X_2), FY_2)+3g(h(JX_1, X_2), FTY_2).
\end{aligned}\end{equation}
If $M$ is ${\mathfrak{D}}\oplus {\mathfrak{D}}^\theta$--mixed totally geodesic, then we find from \eqref{3.14} that
 \begin{align*}
(\cos^2\theta-9)JX_1(\ln f_2)=0,
\end{align*}
which implies that either $\cos\theta=\pm3$, which is not possible or $JX_1(\ln f_2)=0$, i.e., $f_2$ is constant. This completes the proof.
 \end{proof}
 
 \begin{remark}\label{R1} Theorems \ref{T1} and \ref{T2} imply that a proper bi-warped product submanifold $M=M_T\times_{f_1}\! M_\perp\times_{f_2}\! M_\theta$ in a nearly Kaehler manifold is neither ${\mathfrak{D}}\oplus {\mathfrak{D}}^\perp$--mixed totally geodesic nor ${\mathfrak{D}}\oplus{\mathfrak{D}}^\theta$--mixed totally geodesic.
 \end{remark}


\section{Inequality for the second fundamental form}

Let $M=M_T\times_{f_1}\! M_\perp\times_{f_2}\!M_\theta$ be an $n$-dimensional proper bi-warped product submanifold of a nearly Kaehler manifold $\tilde M^{2m}$. We consider a local orthonormal frame field $\{e_1,\ldots,e_{n}\}$ of $TM$ such that
 \begin{align*}
&{\mathfrak{D}}={\rm Span}\{e_1,\cdots,\,e_t,\,e_{t+1}=Je_1,\cdots,e_{2t}=Je_t\},\\
&{\mathfrak{D}}^\perp={\rm Span}\{e_{2t+1}=\hat e_1,\cdots,e_{2t+p}=\hat e_p\},\\
&{\mathfrak{D}}^\theta={ \rm Span}\{e_{2t+p+1}=e^*_1,\cdots,e_{2t+p+q}=e^*_t,
\\&\hskip.8in  e_{2t+p+q+1}=\sec\theta e^*_1,\cdots,e_{n}=\sec\theta e^*_q\}.
\end{align*}
Then $\dim M_T=2t,\,\,\dim M_\perp=p$ and $\dim M_\theta=2q$. Moreover, the orthonormal frame fields $E_1,\ldots, E_{2m-n-p-2q}$ of the normal subbundle $T^\perp M$ are given by
 \begin{align*}
&J{\mathfrak{D}}^\perp={\rm Span}\{E_1=J\hat e_1,\cdots,E_p=J\hat e_p\},\\
&F{\mathfrak{D}}^\theta={ \rm Span}\{E_{p+1}=\csc\theta Fe^*_1,\cdots,E_{p+q}=\csc\theta Fe^*_p,\,
\\&\hskip.7in  E_{p+q+1}=\csc\theta\sec\theta FTe^*_1,\cdots,E_{p+2q}=\csc\theta\sec\theta FTe^*_q\},\\
&\mu={\rm Span}\{E_{p+2q+1},\cdots,\,E_{2m-n-p-2q}\}.
\end{align*}

The main result of this article is the following sharp inequality for  bi-warped product submanifolds in a nearly Kaehler manifold.

\begin{theorem}\label{T3}
Let $M=M_T\times_{f_1}\! M_\perp\times_{f_2}\! M_\theta$ be a bi-warped product submanifold of a nearly Kaehler manifold $\tilde M$, where $M_T,\,M_\perp$ and $M_\theta$ are holomorphic, totally real and proper slant submanifolds of $\tilde M$, respectively. Then we have:
  \begin{enumerate}
 \item[(i)] The second fundamental form $h$ and the warping functions $f_1,\,f_2$ satisfy
 \begin{align}\label{main}
 \|h\|^2\geq 2p\|\nabla(\ln f_1)\|^2+4q\left(1+\frac{10}{9}\cot^2\theta\right)\|\nabla(\ln f_2)\|^2
 \end{align}
 where $p=\dim M_\perp,\,q=\frac{1}{2}\dim M_\theta$ and $\nabla(\ln f_i)$ is the gradient of $\ln f_i$.
 
\item[(ii)] If the equality sign in \eqref{main} holds identically, then $M_T$ is totally geodesic in $\tilde M$, and $M_\perp, M_\theta$ are totally umbilical in $\tilde M$. Moreover, $M$ is neither ${\mathfrak{D}}\oplus{\mathfrak{D}}^\perp$--mixed totally geodesic nor ${\mathfrak{D}}\oplus{\mathfrak{D}}^\theta$--mixed totally geodesic in $\tilde M$.
\end{enumerate}
\end{theorem}
\begin{proof} From the definition of $h$, we have
\begin{align*}
\|h\|^2=\sum_{i,j=1}^ng(h(e_i, e_j), h(e_i, e_j)) =\sum_{r=1}^{2m-n-p-2q}\sum_{i,j=1}^ng^2(h(e_i, e_j), E_r).\end{align*}
Then we decompose the above relation for the normal subbundles as follows
\begin{equation}\begin{aligned}
 \label{t11} \|h\|^2=\, &\sum_{r=1}^{p}\sum_{i,j=1}^ng^2(h(e_i, e_j), J\hat e_r)+\sum_{r=p+1}^{p+2q}\sum_{i,j=1}^ng^2(h(e_i, e_j), E_r)
\\&+\sum_{r=p+2q+1}^{m-n-p-2q}\sum_{i,j=1}^ng^2(h(e_i, e_j), E_r).
\end{aligned}\end{equation}
Leaving the last $\mu$-components term in \eqref{t11} and using the frame fields of tangent and normal subbundles of $M$, we derive
\begin{align}\notag
\|h\|^2\geq\, &\sum_{r=1}^{p}\sum_{i,j=1}^{2t}g^2(h(e_i, e_j), J\hat e_r)+2\sum_{r=1}^{p}\sum_{i=1}^{2t}\sum_{j=1}^{p}g^2(h(e_i, \hat e_j), J\hat e_r)\notag\\
&+\sum_{r=1}^{p}\sum_{i,j=1}^{p}g^2(h(\hat e_i, \hat e_j), J\hat e_r)+2\sum_{r=1}^{p}\sum_{i=1}^{2t}\sum_{j=1}^{2q}g^2(h(e_i, e^*_j), J\hat e_r)\notag\\
&+\sum_{r=1}^{p}\sum_{i,j=1}^{2q}g^2(h(e^*_i, e^*_j), J\hat e_r)+2\sum_{r=1}^{p}\sum_{i=1}^{2q}\sum_{j=1}^{p}g^2(h(e^*_i, \hat e_j), J\hat e_r)\notag\\
&+\csc^2\theta\sum_{r=1}^{q}\sum_{i,j=1}^{2t}\left[g^2(h(e_i, e_j), Fe^*_r)+\sec^2\theta\,g^2(h(e_i, e_j), FTe^*_r)\right]\notag\\
&\label{t12} +2\csc^2\theta\sum_{r=1}^{q}\sum_{i=1}^{2t}\sum_{j=1}^{p}\left[g^2(h(e_i, \hat e_j), Fe^*_r)+\sec^2\theta g^2(h(e_i, \hat e_j), FTe^*_r)\right]\\
&+\csc^2\theta\sum_{r=1}^{q}\sum_{i,j=1}^{p}\left[g^2(h(\hat e_i, \hat e_j), Fe^*_r)+\sec^2\theta\,g^2(h(\hat e_i, \hat e_j), FTe^*_r)\right]\notag\\
&+2\csc^2\theta\sum_{r=1}^{q}\sum_{i=1}^{p}\sum_{j=1}^{2q}\left[g^2(h(\hat e_i, e^*_j), Fe^*_r)+\sec^2\theta\,g^2(h(\hat e_i,  e^*_j), FTe^*_r)\right]\notag\\
&+\csc^2\theta\sum_{r=1}^{q}\sum_{i,j=1}^{2q}\left[g^2(h(e^*_i, e^*_j), Fe^*_r)+\sec^2\theta\,g^2(h(e^*_i, e^*_j), FTe^*_r)\right]\notag\\
&\notag +2\csc^2\theta\sum_{r=1}^{q}\sum_{i=1}^{2t}\sum_{j=1}^{2q}\left[g^2(h(e_i, e^*_j), Fe^*_r)+\sec^2\theta\,g^2(h(e_i,  e^*_j), FTe^*_r)\right].
\end{align}
We have no relation for warped products for the third, fifth, sixth, ninth, tenth and eleventh terms in \eqref{t12}, therefore, we leave these positive terms. Moreover, by using Lemma \ref{L1} and Lemma \ref{L2} with the relations \eqref{3.9}-\eqref{3.13}, we find that
\begin{align*}
\|h\|^2\geq\, & 2p\sum_{i=1}^{t}\left[\left(-Je_i(\ln f_1)\right)^2+\left(e_i(\ln f_1)\right)^2\right]\\&+4q\csc^2\theta\sum_{i=1}^{t}\left[\left(-Je_i(\ln f_2)\right)^2+\left(e_i(\ln f_2)\right)^2\right]\\
&+\frac{4q}{9}\cot^2\theta\sum_{i=1}^{t}\left[\left(-Je_i(\ln f_2)\right)^2+\left(e_i(\ln f_2)\right)^2\right]\\
=\,&2p\sum_{i=1}^{2t}\left(e_i(\ln f_1)\right)^2+4q\left(\csc^2\theta+\frac{1}{9}\cot^2\theta\right)\sum_{i=1}^{2t}\left(e_i(\ln f_2)\right)^2.
\end{align*}
Then we find the required inequality from the definition of gradient. 
 \vskip.1in
 
For the equality case, we have from the leaving third term in \eqref{t11} that
\begin{align}\label{eq1}
h(TM, TM)\perp\mu
\end{align}
From the vanishing first term and leaving seventh term in \eqref{t12}, we find
\begin{align}\label{eq2}
h({\mathfrak{D}}, {\mathfrak{D}})\perp J{\mathfrak{D}}^\perp\,\,\,{\mbox{and}}\,\,\,h({\mathfrak{D}}, {\mathfrak{D}})\perp F{\mathfrak{D}}^\theta.
\end{align}
Then we find from \eqref{eq1} and \eqref{eq2} that
\begin{align}\label{eq3} h({\mathfrak{D}}, {\mathfrak{D}})=0.\end{align}

On the other hand, from the leaving third and ninth terms in \eqref{t12}, we get
\begin{align}\label{eq4}
h({\mathfrak{D}}^\perp, {\mathfrak{D}}^\perp)\perp J{\mathfrak{D}}^\perp\,\,\,{\mbox{and}}\,\,h({\mathfrak{D}}^\perp, {\mathfrak{D}}^\perp)\perp F{\mathfrak{D}}^\theta.
\end{align}
Again,  we conclude from  \eqref{eq1} and \eqref{eq4} that
\begin{align}\label{eq5}
h({\mathfrak{D}}^\perp, {\mathfrak{D}}^\perp)=0.
\end{align}
Also, from the leaving fifth and eleventh terms in the right hand side of \eqref{t12}, we have
\begin{align}\label{eq6}
h({\mathfrak{D}}^\theta, {\mathfrak{D}}^\theta)\perp J{\mathfrak{D}}^\perp\,\,\,{\mbox{and}}\,\,h({\mathfrak{D}}^\theta, {\mathfrak{D}}^\theta)\perp F{\mathfrak{D}}^\theta.
\end{align}
Then we obtain from \eqref{eq1} and \eqref{eq6} that
\begin{align}\label{eq7}
h({\mathfrak{D}}^\theta, {\mathfrak{D}}^\theta)=0.
\end{align}
Moreover, from the leaving sixth and tenth terms in \eqref{t12}, we get
\begin{align}\label{eq8}
h({\mathfrak{D}}^\perp, {\mathfrak{D}}^\theta)\perp J{\mathfrak{D}}^\perp\,\,\,{\mbox{and}}\,\,h({\mathfrak{D}}^\perp, {\mathfrak{D}}^\theta)\perp F{\mathfrak{D}}^\theta.
\end{align}
Therefore, from \eqref{eq1} and \eqref{eq8} we obtain
\begin{align}\label{eq9}
h({\mathfrak{D}}^\perp, {\mathfrak{D}}^\theta)=0.
\end{align}

On the other hand, from the vanishing eighth term in \eqref{t12} with\eqref{eq1}, we have
\begin{align}\label{eq10}
h({\mathfrak{D}}, {\mathfrak{D}}^\perp)\subset J{\mathfrak{D}}^\perp.
\end{align}
Similarly, from the vanishing forth term in \eqref{t12} with \eqref{eq1}, we get
\begin{align}\label{eq11}
h({\mathfrak{D}}, {\mathfrak{D}}^\theta)\subset F{\mathfrak{D}}^\theta.
\end{align}
Since $M_T$ is totally geodesic in $\tilde M$ (see, e.g., \cite{Bi,C3}), using this fact together with \eqref{eq3}, \eqref{eq5} and \eqref{eq9}, we know $M_T$ is totally geodesic in $\tilde M$. Also, since $M_\perp$ and $M_\theta$ are totally umbilical in $M$, using this fact together with \eqref{eq5}, \eqref{eq7}, \eqref{eq10} and \eqref{eq11}, we conclude that $M_\perp$ and $M_\theta$ are both totally umbilical in $\tilde M$. Furhter, it follows from Remark \ref{R1}, \eqref{eq10} and \eqref{eq11} that $M$ is neither ${\mathfrak{D}}\oplus{\mathfrak{D}}^\perp$--mixed totally geodesic nor ${\mathfrak{D}}\oplus{\mathfrak{D}}^\theta$--mixed totally geodesic in $\tilde M$. Consequently, the theorem is proved completely.
\end{proof}

\section{Some applications}

Theorem \ref{T3} implies the following.

\begin{theorem}\label{T:5.1} {\rm \cite{C3}} Let $M=M_T\times_f M_\perp$ be a  $CR$-warped product in a
Kaeahler manifold $\tilde M$. Then the second fundamental form $h$ of $M$ satisfies 
 \begin{align}\label{5.1}||h||^2\geq 2p\,||\nabla (\ln f)||^2, \end{align} 
 where $p=\dim M$.
 Moreover, if the equality sign of \eqref{5.1} holds identically, then  $M_T$ is  totally geodesic  and $M_\perp $ is  totally umbilical in $\tilde M$. 
\end{theorem}

A warped submanifold of the form $M=M_T\times_{f} M_\theta$  in a a nearly Kaehler manifold $\tilde M$ is called {\it semi-slant} if $M_T$ is a holomorphic submanifold and $M_\theta$ is a proper slant submanifold in $\tilde M$.

The next result was proved in \cite{KK09}.

\begin{theorem}\label{T:5.2}  Let $M_T\times_f M_\theta$ be a semi-slant warped product of a nearly K\"ahler manifold $\tilde M$. Then the second fundamental form $h$ of $M$ satisfies 
 \begin{align}\label{5.2}||h||^2\geq 4q \csc^2\theta \left\{1+\frac{1}{9}\cos^4\theta \right\}|\nabla (\ln f)|^2.\end{align} 
\end{theorem}

On the other hand, Theorem \ref{T3} implies the following.

\begin{theorem}\label{T:5.3} {\rm \cite{AKU17}}
Let $M=M_T\times_{f} M_\theta$ be a semi-slant warped product submanifold of a nearly Kaehler manifold $\tilde M$. Then second fundamental form $h$ and the warping function $f$ satisfy
 \begin{align}\label{5.3}
 \|h\|^2\geq 4q\left\{1+\frac{10}{9}\cot^2\theta\right\}\|\nabla(\ln f)\|^2.
 \end{align}
 Moreover,  if the equality sign in \eqref{main} holds identically, then $M_T$ is totally geodesic and  $M_\theta$ are totally umbilical in $\tilde M$. 
 \end{theorem}

\begin{remark} Theorem \ref{T:5.3} improves Theorem \ref{T:5.2} since $$9+10\cot^2\theta > \csc^2\theta (9+\cos^4\theta)$$ holds for every $\theta\in (0,\frac{\pi}{2})$. Furthermore, Theorem \ref{T:5.3} shows that inequality \eqref{5.2} in Theorem \ref{T:5.2} is not sharp.
\end{remark}

\vskip.15in
\noindent {\it Acknowledgements.}
This project was funded by the Deanship of Scientific Research (DSR), King Abdulaziz University, Jeddah, Saudi Arabia under grant no. (KEP-PhD-33-130-38). Therefore, the authors acknowledge their thanks to the DSR technical and financial support.

\end{document}